\documentclass[11pt]{amsart}
\usepackage{amsmath}
\usepackage{amsthm}
\usepackage{amssymb}
\usepackage{amscd}
\usepackage{hyperref}
\usepackage{tikz}

\usepackage{enumerate}
\usetikzlibrary{matrix,arrows}

\newcommand{\Spec}{\operatorname{Spec}}
\newcommand{\length}{\ell}
\newcommand{\loewy}{\ell\ell}

\newcommand{\eh}{\operatorname{e}}

\newcommand{\mf}{\mathfrak}

\DeclareMathOperator{\ord}{ord}

\DeclareMathOperator{\fiber}{\mathfrak F}
\DeclareMathOperator{\lec}{lech}

\DeclareMathOperator{\thj}{c}

\newcommand{\sO}{\mathcal O}

\newcommand{\m}{\mf m}
\newcommand\ceil[1]{\lceil#1\rceil}
\newcommand\bceil[1]{\Bigl\lceil#1\Bigr\rceil}

\newtheorem{theorem}{Theorem}
\newtheorem{lemma}[theorem]{Lemma}
\newtheorem{proposition}[theorem]{Proposition}
\newtheorem{corollary}[theorem]{Corollary}

\newtheorem{conjecture}[theorem]{Conjecture}
\newtheorem*{statement*}{Statement}
\newtheorem*{theorem*}{Theorem}
\newtheorem*{lemma*}{Lemma}
\newtheorem*{fact*}{Fact}
\newtheorem{question}[theorem]{Question}

\theoremstyle{definition}
\newtheorem{definition}[theorem]{Definition}
\newtheorem*{definition*}{Definition}
\newtheorem{example}[theorem]{Example}
\newtheorem*{example*}{Example}
\newtheorem{inequality}[theorem]{Inequality}

\theoremstyle{remark}
\newtheorem{remark}[theorem]{Remark}
\newtheorem{observation}[theorem]{Observation}
\newtheorem*{acknowledgement}{Acknowledgments}

\numberwithin{theorem}{section}
\numberwithin{equation}{section}

\begin{document}

\title[The multiplicity and the number of generators of integrally closed ideals]
{The multiplicity and the number of generators of an integrally closed ideal}
\author{Hailong Dao}
\address{Department of Mathematics\\
University of Kansas\\
 Lawrence, KS 66045-7523 USA}
\email{hdao@math.ku.edu}

\author{Ilya Smirnov}
\address{
Department of Mathematics\\
University of Michigan\\
Ann Arbor, MI 48109, USA
}
\email{ismirnov@umich.edu}

\maketitle

\begin{abstract}
Let $(R, \mf m)$ be a Noetherian local ring and $I$ a $\m$-primary ideal. In this paper, we study an inequality involving 
the number of generators, the Loewy length and the multiplicity of $I$. There is strong evidence that the inequality holds for all integrally closed ideals of finite
colength if and only if $\Spec R$ has sufficiently nice singularities. We verify the inequality for regular local rings in all dimensions, for rational singularity in dimension $2$, and cDV 
singularities in dimension $3$. In addition, we can classify when the inequality always hold for a Cohen-Macaulay $R$ of dimension at most two.  We also discuss relations to various topics:  classical results on rings with minimal multiplicity and rational singularities, the recent work on $p_g$ ideals by Okuma-Watanabe-Yoshida, a conjecture of Huneke, Musta\c{t}\v{a}, Takagi, and Watanabe on $F$-threshold,  multiplicity of the fiber cone, and  the $h$-vector of the associated graded ring. 
\end{abstract}

\section{Introduction}

Let $(R,\mf m)$ be a Noetherian local ring of dimension $d\geq 1$ and $I$ be an integrally closed, $\m$-primary ideal. Let $\mu(I)$ and $\eh(I)$ be the minimal number of generators and the Hilbert-Samuel multiplicity of $I$, respectively. 
With a slight abuse of notation we will use $\loewy(I)$ to denote the Loewy length of the artinian ring $R/I$, i.e., 
the smallest integer $n$ such that $\mf m^n \subseteq I$.

The purpose of this note is to study the following simple inequality
\begingroup
\newcounter{tmp}
\setcounter{tmp}{\value{theorem}}
\setcounter{theorem}{0} 
\renewcommand\thetheorem{A}
\begin{inequality}\label{ineq}
\[
(d-1)!(\mu(I) - d + 1) \loewy(I)  \geq \eh(I)
\]
\end{inequality}
\endgroup
\setcounter{theorem}{\thetmp}
that connects these invariants of $I$. 

It is easy to see that when one takes $I=\m^c$ and lets $c$ goes to infinity, the ratio between left side and right side in the above inequality goes to $1$. So asymptotically, this is the best we can hope for. On the other hand, it is clear that inequality \ref{ineq} does not always hold. For example, when $1\leq d\leq 2$ and $I=\m$, the inequality implies that $R$ has minimal multiplicity.

As our first main result, we show that inequality~\ref{ineq} holds in every regular local ring. 
To do so, in Theorem~\ref{general mixed} we establish a Lech-type inequality on
the number of generators which is of independent interest. 

\begin{theorem}\label{mainTheorem}
Let $(R,\m)$ be a local ring of dimension $d\geq 1$. 
Let  $I$ be an integrally closed $\m$-primary  ideal.
Then
\begin{enumerate}
\item for a general element $x \in \m$, we have
$
(d-1)!\eh(R) (\mu(I) - d + 1) \geq \eh(IR/(x)), 
$
\item
$
(d-1)!\eh(R) \loewy(I) (\mu(I) - d + 1) \geq \eh(I).
$
\end{enumerate}
 In particular, inequality \ref{ineq} always holds if $R$ is regular.
\end{theorem}

Combining this with a result of D'Cruz and Verma (\cite{DCruzVerma2}), who showed that in a Cohen-Macaulay ring
$\mu(I) - d+ 1 \leq \eh(IR/(x))$, we obtain a parallel with the classical inequalities on lengths:
\[
\length (R/I) \leq \eh(I) \leq d! \eh(R) \length (R/I).
\]

Theorem~\ref{mainTheorem} shows that every local ring satisfies inequality~\ref{ineq}
weakened by $\eh(R)$, a factor representing how singular is $R$.
Thus, we expect inequality~\ref{ineq} to hold for all integrally closed $\m$-primary ideals as long as $R$ has ``nice" singularities. 
In this note we investigate this in small dimensions. The following summarizes our main  results (for details, see Theorem \ref{dimension one}, Corollary \ref{unramified criterion} and Corollary \ref{cool evidence}):

\begin{theorem}\label{summary}
Let $(R, \m)$ be a local ring. 
\begin{enumerate}
\item If  $R$ is Cohen-Macaulay and $d=1$, inequality \ref{ineq} holds for all integrally closed, $\m$-primary ideals if and only if $R$ has minimal multiplicity. 
\item If  $R$ is Cohen-Macaulay and $d=2$, inequality \ref{ineq} holds for all integrally closed, $\m$-primary ideals if $R$ has minimal multiplicity and $\m$ is normal (i.e., all powers of $\m$ are integrally closed).  The converse holds if $R$ is analytically unramified.  In particular,  inequality \ref{ineq} holds for all integrally closed, $\m$-primary ideals if $R$ has an isolated rational singularity.
\item When $d=3$,  inequality \ref{ineq} holds for all integrally closed, $\m$-primary ideals if  $R$ has cDV singularity (that is, for a general $x\in \mf m$, $R/(x)$ is an ADE singularity). 
\end{enumerate}
\end{theorem}
 
We also discuss and prove many related results, such as sharper and similar inequalities in the dimension $2$ cases, as well as inequalities for multiplicity of fiber cones. 

It is worth noting that our investigation and some of the results were inspired by other works  with rather  different flavors. Both sides of inequality~\ref{ineq} appears as upper bounds for dimension of the {\it singularity category} of a  Cohen-Macaulay ring $R$ (when $d\leq 2$) in the work of the first author and Takahashi (\cite{DaoTakahashi}). It is thus natural to ask when such inequality always holds. Also, in dimension $2$ some of our results can be proved by recent beautiful work on $p_g$ ideals by Okuma, Watanabe and Yoshida (\cite{OWY1, OWY2}). Finally, it turned out that the regular case in dimension two of Inequality \ref{ineq} implies a special case of a conjecture by  Huneke, Musta\c{t}\v{a}, Takagi, and Watanabe on $F$-threshold.  

As stated above, we expect the inequality~\ref{ineq} to hold when $R$ has nice singularities. Just to  be more focused, we would like to propose:
\begin{question}\label{maybe stupid}
Does inequality~\ref{ineq} hold for all integrally closed, $\m$-primary ideals if $R$ has  Gorenstein terminal singularities? 
\end{question}

The paper is organized as follows. Section \ref{prelim} deals with preparatory materials, including a summary on $\m$-full ideals. In the next section we prove inequality \ref{ineq} for regular local rings. The result follows from a more general inequality for any local rings, see Theorem \ref{regular} and Remark \ref{rmLech}. Section \ref{smdim} handles the cases of dimension at most $3$. We also present a few examples to illustrate the sharpness of our results. Section \ref{alternative} discusses special cases in dimension two where one can  sometimes obtain sharper inequalities and alternative proofs, including the geometric approach used in \cite{OWY1, OWY2}. In Section \ref{threshold} we prove a special case of the conjecture by Huneke, Musta\c{t}\v{a}, Takagi, and Watanabe on $F$-threshold.  Section \ref{fiber} analyzes the asymptotic version of our inequality, which leads to some new inequalities about multiplicity of the fiber cones. Finally in Section \ref{powers} we study the inequality for powers of maximal ideals, and specify the stringent conditions it forces on the singularity of $R$ via the $h$-vector of the associated graded ring.   

\begin{acknowledgement}
We are deeply grateful to Kei-ichi Watanabe for many detailed discussions on \cite{OWY1, OWY2}. His insights gave us the geometric proofs in Section \ref{alternative} and suggested the if and only if statement in the dimension two case of Theorem \ref{mainTheorem}. We are also grateful to  Craig Huneke, Dan Katz, Jonathan Monta\~no, Mircea Musta\c{t}\v{a},   Shunsuke Takagi, Bernd Ulrich, Junzo Watanabe and Ken-ichi Yoshida for helpful discussions. We thank the Mathematics Department at University of Kansas and University of Michigan, where most of this work was done,  for hospitable working conditions. The first author is partially  supported by NSA grant H98230-16-1-001. 

\end{acknowledgement}

\section{Preliminaries}\label{prelim}

One of our key observations is that it is sometimes easier to work with $\m$-full ideals instead of integrally closed ideals. We recall:

\begin{definition}\label{defs}
Let $(R, \mf m)$ be a local ring with infinite residue field. An ideal $I$ of $R$ is $\mf m$-{\it full} if there exists an element $z \in \mf m$  such that 
$\mf mI : z = I$ (equivalently, for any general element $z$). $I$ is {\it full} if  there exists an element $z \in \mf m$ such that $I:z=I:\mf m$.   

In general, $I$ is $\mf m$-full if there exists a faithfully flat extension $R\to S$ with infinite residue field such that $\mf mS$ is the maximal ideal of $S$ and $IS$ is $\mf mS$-full. 
\end{definition}

\begin{remark}\label{pass m-full}
The following facts are known: 

\begin{enumerate}
\item If  $I$ is integrally closed, then $I$ is $\mf m$-full  or $I = \sqrt{(0)}$ (\cite[Theorem~2.4]{Goto}).
\item If $I$ is $\mf m$-primary and $\mf m$-full, then $\mu(J) \leq \mu(I)$ for any ideal $J\supseteq I$ (\cite[Theorem~3]{WatanabeJ}).  
\item $\mf m$-full ideals are full (\cite[Lemma~1]{WatanabeJ}). 
\end{enumerate}
\end{remark}

\begin{remark}\label{pass to infinite residue field}
By passing to $S = R[x]_{\mf m[x]}$, we can harmlessly assume that the residue field of $R$ is infinite.
Then $IS$ is integrally closed if and only if $I$ is integrally closed and  
this extension does not change the number of generators, multiplicity, or the Loewy length. 
\end{remark}

It is not hard to see that there is no difference in establishing inequality ~\ref{ineq} for $\m$-full ideals or for integrally closed ideals. 

\begin{proposition}\label{go mfull}
Let $(R, \mf m)$ be a local ring of dimension at least $1$. The following are equivalent:
\begin{enumerate}
\item Inequality~\ref{ineq} holds for all $\mf m$-primary, integrally closed ideals. 
\item Inequality~\ref{ineq} holds for all $\mf m$-primary, $\mf m$-full ideals. 
\end{enumerate} 
\end{proposition}

\begin{proof}
Assume $(1)$ and let $I$ be $\mf m$-full. Since $I \subseteq \overline{I}$ and $I$ is $\m$-full, 
we must have  that $\mu (I) \geq \mu (\overline{I})$ by Remark \ref{pass m-full} and $\loewy (I) \geq \loewy(\overline{I})$.
Now, inequality \ref{ineq} holds for $\overline I$ and therefore for $I$ if we recall that $\eh(I) = \eh(\overline {I})$.

Assume $(2)$ and let $I$ be integrally closed $\mf m$-primary ideal. Then as $\dim R>0$, $I$ must be $\mf m$-full by Remark \ref{pass m-full}.

\end{proof}

We also record the following elementary lemma that will be used later.

\begin{lemma}\label{filtration}
Let $(R,\mf m)$ be a local ring and $I$ be a $\mf m$-primary ideal. Let $x,y \in \mf m$. Then 
\[
\eh\left(IR/(xy)\right) \leq \eh\left(IR/(x) \right) +\eh\left( IR/(y) \right).
\]
\end{lemma}

\begin{proof}
It is enough to prove that for every ideal $J$
\[\length (R/(J,xy)) \leq \length (R/(J,x)) + \length (R/(J,y)).\] 
For, if it holds, then we just need take $J=I^n$ and divide by $n^{\dim R/(xy)}$ to get the desired conclusion. Note that $\dim R/(xy) \geq \max\{\dim R/(x), \dim R/(y)\}$ as $R/(xy)$ surjects onto $R/(x), R/(y)$.  

Consider the short exact sequence 
\[0 \to (x)/(xy) \to R/(xy) \to R/(x) \to 0.\] 
The claim now easily follows after tensoring it with $R/J$ and
and observing that $R/(y)$ surjects onto $(x)/(xy)$. 
\end{proof}

\section{A general inequality and inequality \ref{ineq} for regular local rings}\label{regular}

In this section we will show that inequality \ref{ineq} holds for a regular local ring of arbitrary dimension. 
In fact we will prove a more general  inequality for an arbitrary local ring, utilizing the theory of $\mf m$-full ideals discussed in the last section.  

\begin{theorem}\label{general mixed}
Let $(R, \mf m)$ be a local ring of dimension $d>0$ and $I$ be an $\mf m$-primary $\mf m$-full ideal.
Then for a general element $x \in \mf m$
\[
(d-1)!\eh(R) (\mu(I) - d + 1) \geq \eh(IR/(x)).
\]
\end{theorem}
\begin{proof}
We can assume that $R$ has an infinite residue field (Remark~\ref{pass to infinite residue field}). Let $x$ be a general element such that $I\mf m:x=I$. 
Furthermore, a general element can be taken to be a superficial parameter with respect to $\mf m$ 
(for example, see \cite[Proposition~8.5.7]{HS}), so we may assume that $\dim R/(x) = d-1$ and $\eh(R/(x)) = \eh (R)$. 
By \cite[Theorem~2]{WatanabeJ} we have
\[
\mu (I) = \mu (IR/(x)) + \length (R/(I,x)).
\]

Next observe that $\mu (IR/(x)) \geq d - 1$, 
because $IR/(x)$ is a $0$-dimensional ideal in the ring $R/(x)$ of dimension $d - 1$. 
Thus it will suffice to show that
\[
(d-1)!\eh(R) \length (R/(J,x)) \geq \eh(JR/(x))
\]
for {\it any} $\mf m$-primary ideal $J$ of $R$.
But this is nothing else than 
Lech's inequality (\cite[Theorem~3]{Lech}) applied to $JR/(x)$ in $R/(x)$,
since $\eh(R) = \eh(R/(x))$ by superficiality of $x$.
\end{proof}

\begin{corollary}\label{loewy general}
Let $(R, \mf m)$ be a local ring of dimension $d>0$ and $I$ be an $\mf m$-primary $\mf m$-full ideal.
Then 
\[
(d-1)!\eh(R) \loewy(I) (\mu(I) - d + 1) \geq \eh(I).
\]
\end{corollary}
\begin{proof}
It is left to show that 
\[
\loewy(J) \eh (JR/(x)) \geq \eh(J).
\]
Observe that $x^{\loewy (J)} \in J$, so by  \cite[Proposition~28.1]{HIO}  
\[\eh(JR/(x^{\loewy (J)}))\geq \eh(JR)\]
and the claim follows since
$\eh (JR/(x^{\loewy (I)})) \leq \loewy(I) \eh (JR/(x))$ by Lemma \ref{filtration}.
\end{proof}

\begin{corollary}\label{reg case}
Inequality~\ref{ineq} holds for $\mf m$-full ideals in a regular local ring of positive dimension. 
\end{corollary}

\begin{remark}\label{rmLech}

For a local ring $(R,\mf m)$ of dimension $d$ we can define the Lech constant of $R$, $\lec(R)$ as the smallest value $c$ such that $cd!\length (R/J) \geq  \eh(J)$ for any $\mf m$-primary ideal $J$.  Lech's result shows that $\lec(R)\leq \eh(R)$. 
By looking at the asymptotic behavior, it is not hard to see  that $\lec(R)\geq 1$.  

Our  proof of Corollary~\ref{loewy general} really shows the following. Let $C$ be such that $C\geq \lec(R/(x))$ for a general $x \in \mf m$.  For any $\mf m$-primary integrally closed ideal $I$,
\[
C(d-1)!\loewy(I) (\mu(I) - d + 1) \geq \eh(I).
\] 

On the other hand, if $R$ is Cohen-Macaulay and $d\leq 1$, it is not hard to see that $\lec(R)=1$ if and only if $R$ is regular, so to get the best possible result for singular rings we need to use other methods. We shall do so in the next section.  
\end{remark}

\section{Small dimensions}\label{smdim}

In dimension at most two and $R$ is Cohen-Macaulay,  inequality~\ref{ineq} forces $R$ to have minimal multiplicity by letting $I= \mf m$  and using  Abhyankar's inequality 
(\cite[(1)]{Abhyankar}). In this section we will carefully study what happens in rings of dimension at most two and give necessary and sufficient conditions such that that inequality~\ref{ineq} always holds  for  $\mf m$-primary integrally closed ideals. In dimension three, we are able to show  that the inequality always hold if $R$ has a compound Du Val singularity. 

\subsection{Dimension one}

\begin{theorem}\label{dimension one}
Let $(R, \mf m)$ be a Cohen-Macaulay local ring of dimension one. Then the inequality 
\[
\mu(I) \loewy(I) \geq \eh(I)
\]
holds for every integrally closed ideal $I$ if and only if $R$ has minimal multiplicity, i.e., $\eh(R) = \mu (\m)$.
\end{theorem}
\begin{proof}
To see that $R$ must have minimal multiplicity we apply the inequality to $I = \mf m$. 

For the other direction, observe that $\eh(I) = \length(I^n/I^{n+1})$ for $n\gg 0$. 
Let $k = \loewy (I)$ then $\mf m^kI^n \subseteq I^{n+1}$ for all $n$. Therefore
\[
\length(I^n/I^{n+1}) \leq \length (I^n/\mf m^k I^n) = 
\length (I^n/\mf m I^n) + \cdots + \length (\mf m^{k-1} I^n / \mf m^k I^n)
= \mu(I^n) + \mu(\mf m I^n) + \cdots + \mu (\mf m^{k-1} I^n).
\]

By Remark~\ref{pass to infinite residue field} we may assume that $R$ has an infinite residue field.
Let $x$ be a minimal reduction of $\mf m$ and observe that for any $\mf m$-primary ideal $J$
\[
\mu (J) = \length (J/\mf m J) \leq \length (J/xJ) = \eh (x, J) = \eh (\mf m, J)
\] 
since $J$ is a Cohen-Macaulay submodule of $R$. Since $R/J$ has dimension $0$, it follows that
$\eh (\mf m, J) = \eh (\mf m, R) = \mu(\mf m)$. 

The above inequality on the number of generators gives that
\[
\length(I^n/I^{n+1}) \leq k \mu(\mf m).
\] 
By Remark~\ref{pass m-full},  $\mu (I) \geq \mu (\mf m)$ and the claim follows.

\end{proof}

\begin{remark}
The proof shows that for any $\mf m$-primary ideal $J$ then $\mu(J) \leq \mu(\mf m)$. 
Since $\mu (\mf m) \leq \eh(R)$ by Abhyankar's inequality, this observation strengthens the classical result of 
Akizuki-Cohen (see also \cite[Corollary~2.3]{DCruzVerma}).

Also, if $J$ is $\mf m$-full, then  $\mu(J)=\eh(R)$ in any one dimensional local Cohen-Macaulay ring with minimal multiplicity.
\end{remark}

\subsection{Dimension 2}

In dimension two, we shall show that when $R$ is analytically unramified, then inequality \ref{ineq} holds for all $\mf m$-primary, integrally closed ideals if and only if $\mf m$ is normal  and $R$ has minimal multiplicity. We employ the techniques developed by David Rees (\cite{ReesRat}),
which  allow us to exploit properties of  mixed multiplicities.

\begin{definition}
Let $(R, \mf m)$ be a two-dimensional local ring and $I, J$ be $\mf m$-primary ideals. 
Following \cite{ReesSharp}
the mixed multiplicity of $I$ and $J$, $\eh (I \mid J)$  is defined by the equation
\[
\eh (I^rJ^s) = \eh(I)r^2 + 2\eh(I \mid J) rs + \eh(J) s^2.
\]
If the residue field of $R$ is infinite, then $\eh(I \mid J) = R/(a,b)$ where $a \in I, b \in J$ are general elements (\cite{Teissier}).
It follows from the definition that $\eh (I^n \mid J) = n \eh (I \mid J)$ for any integer $n \geq 1$.
\end{definition}

We start with an easy application of the machinery developed by Rees in \cite{ReesRat}.
We encourage the reader to look at \cite{VermaComplete} where a major part of Rees's argument
is rewritten in a more accessible way.

\begin{lemma} \label{Rees lemma}
Let $(R, \mf m)$ be an analytically unramified Cohen-Macaulay local ring of dimension two.
If $J$ is a normal ideal with reduction number $1$ then for any integrally closed ideal $I$
\[
\eh (I \mid J) = \length (R/IJ) - \length (R/I) - \length (R/J).
\]
\end{lemma}
\begin{proof}
First, we can extend the residue field without changing the inequality (Remark~\ref{pass to infinite residue field}).
Observe that $J$ satisfies condition (ii) of \cite[Theorem~2.6]{ReesRat}. Hence, using \cite[Theorems~2.5 and 2.6]{ReesRat}, we obtain that 
\[\overline {IJ} = IJ = bI + aJ\] 
where $a \in I, b \in J$ are general elements.
Now, in the notation of \cite[Theorem~2.1]{ReesRat} with $\mf c = R$ 
and using  the formula 
\[
 \length (R/( bI + aJ)) -\length (R/(a,b))  = \length (N_1) = \length (R/I) + \length (R/\overline{J})
\] 
obtained in its proof, we have that
\[
\eh(I\mid J) := \length (R/(a,b))  = \length (R/IJ) - \length (R/I) - \length (R/J).
\]
\end{proof}

As a corollary we generalize a result of Okuma, Watanabe, and Yoshida (\cite[Theorem~6.1]{OWY1}).

\begin{corollary}
Let  $(R, \mf m)$ be a two-dimensional analytically unramified Cohen-Macaulay local ring with minimal multiplicity.
If $\mf m$ is normal, then for any integrally closed ideal $I$
\[
\mu(I) = \eh (I \mid \mf m) + 1.
\]
\end{corollary}
\begin{proof}
Let $J=\m$ in the previous lemma. 
\end{proof}

\begin{corollary} \label{minimal multiplicity}
Let  $(R, \mf m)$ be a two-dimensional analytically unramified Cohen-Macaulay local ring with minimal multiplicity.
If $\mf m$ is normal, then for any integrally closed ideal $I$
\[
\eh (I) \leq \loewy (R/I)\eh(I \mid \mf m) = \loewy (R/I) (\mu(I) - 1).
\]
\end{corollary}
\begin{proof}
Since $\mf m^{\loewy (R/I)} \subseteq I$ we have
\[
\eh(I) = \eh (I \mid I) \leq \eh (I \mid \mf m^{\loewy(R/I)}) = \loewy(R/I) \eh (I \mid \mf m)
\]
and we may use the preceding corollary.
\end{proof}

We have just obtained a class of analytically unramified rings that satisfy our inequality. The converse also holds, but will require more work.

\begin{lemma}\label{numgen bound}
Let $(R, \mf m)$ be a Cohen-Macaulay local ring of dimension $2$ with infinite residue field. 
Then $\mu (\overline{\mf m^n}) \leq n \eh (R) + 1$. If equality occurs, then $(\overline{\mf m^n}, x) = ({\mf m^n}, x)$ for a general element $x \in \m$.
\end{lemma}
\begin{proof}
Since $\overline{\mf m^n}$ is $\mf m$-full, by a result of Watanabe (\cite{WatanabeJ}) for a general element $x$
\[
\mu (\overline{\mf m^n}) = \length (R/(\overline{\mf m^n}, x)) + \mu (\overline{\mf m^n}R/(x)).
\]
Since $R/(x)$ has dimension $1$, by the proof of Theorem~\ref{dimension one},
$\mu (I) \leq \eh (R/(x)) = \eh(R)$ for any ideal $I$ of $R/(x)$.
Thus we can easily estimate
\[
\length (R/(\overline{\mf m^n}, x)) \leq \length (R/({\mf m^n}, x))= 1 + \mu (\mf mR/(x)) + \cdots + \mu (\mf m^{n-1}R/(x)) \leq 1 + (n-1)\eh(R)
\]
and the desired inequality follows. Equality would force $\length (R/(\overline{\mf m^n}, x)) = \length (R/({\mf m^n}, x))$, thus $(\overline{\mf m^n}, x) = ({\mf m^n}, x)$.
\end{proof}

Combining with the  results of Huneke and Itoh we can prove the converse of Corollary~\ref{minimal multiplicity}.

\begin{theorem}\label{get normal}
Let $(R, \mf m)$ be an analytically unramified Cohen-Macaulay local ring of dimension $2$ and minimal multiplicity. 
If for all $n$
\[
\mu(\overline{\mf m^n}) = \mu (\mf m^n) = n\eh(R) + 1
\]
then $\mf m$ is normal.
\end{theorem}
\begin{proof}
Since we may assume that the residue field is infinite, so 
two general elements $x, y \in \mf m$ will form a minimal reduction of $\mf m$.

By hypothesis $\mu (\overline{\mf m^2}) = 2\eh(R) + 1$, so  Lemma~\ref{numgen bound} shows that 

\[\overline {\mf m^2} \subseteq \mf m^2 + (x) \subseteq (x,y)\]
and, by the Huneke-Itoh theorem (\cite[Theorem~1]{Itoh}), 
\[
\overline {\mf m^2} =  \overline {\mf m^2} \cap (x,y) = (x,y)\mf m \subseteq \mf m^2.
\]
Thus $\mf m^2$ is integrally closed.

It was shown by Itoh in \cite[Proposition~10]{Itoh}, that
\[
\length (R/\overline{\mf m^n}) \leq \eh (R) \binom{n + 1}{2} - 
\left (\eh(R) - 1 + \length \left (\overline{\mf m^2}/(x,y)\mf m \right ) \right) n + 
\length \left (\overline{\mf m^2}/(x,y)\mf m  \right).
\]
Of course, we already know that $\overline{\mf m^2} = \mf m^2 = (x,y)\mf m$, so the above inequality gives us 
that the normal Hilbert polynomial $\overline{P(n)}$ is bounded above by the minimal multiplicity polynomial
$\eh (R) \binom{n + 1}{2} - (\eh(R) - 1)n$. 

Let $\overline{u}_n = \overline{P(n)} - \length (R/\overline{\mf m^n})$.
By a result of Huneke (\cite[Theorem~4.5(iii)]{Huneke}) we know that $\overline{u}_n$ is non-negative and non-increasing.
If we let $n = 2$ and use that
\[\length (R/\overline{\mf m^2}) = \length (R/\mf m^2) = \eh (R) + 2,\] 
then we see that
\[
0 = \eh(R) \binom{3}{2} - (\eh(R) - 1)2 - \eh (R) - 2
\geq \overline{u}_2 \geq 0,
\]
so $\overline{u}_2 = 0$. Thus $\overline{u}_n = 0$ for any $n$, and
by \cite[Theorem~4.4]{Huneke} this implies that $\overline{\mf m^{n+1}} = (x,y) \overline{\mf m^n}$ for all $n \geq 2$
which allows us to show by induction that $\mf m^{n+1}$ is integrally closed.

\end{proof}

\begin{corollary}\label{unramified criterion}
Let $(R, \mf m)$ be an unramified local Cohen-Macaulay ring of dimension $2$. 
Then the following conditions are equivalent.
\begin{enumerate}[(a)]
\item $R$ has minimal multiplicity  and $\mf m$ is normal.
\item $\eh(I) \leq \loewy(R/I) (\mu (I) - 1)$
for every $\mf m$-full $\mf m$-primary ideal $I$.
\item $\mu(I) = \eh(\mf m \mid I) +  1$ for every integrally closed $\mf m$-primary ideal $I$.
\end{enumerate}
\end{corollary}
\begin{proof}
Corollary~\ref{minimal multiplicity} together with Proposition~\ref{go mfull}
shows that $(a)$ implies $(c)$ and $(c)$ implies $(b)$, so we are left to show the last implication $(b)$ implies $(a)$.

First, observe that, for $I = \mf m$, the inequality implies that $R$ has minimal multiplicity. 
Consider $\overline{\mf m^n}$. Clearly, $\loewy (\overline{\mf m^n}) = n$ and $\eh(\overline{\mf m^n}) = n^2 \eh(R)$, 
so we obtain that
\[
n \eh(R) + 1 \leq \mu (\overline{\mf m^n}).
\]
By Lemma~\ref{numgen bound}, $\mu (\overline{\mf m^n}) = n \eh(R) + 1$ and the latter is also equal to $\mu(\mf m^n)$ 
because $R$ has minimal multiplicity. Thus the assertion follows from Theorem~\ref{get normal}.
\end{proof}

\begin{remark}
D'Cruz and Verma have proved (\cite[Theorem~2.2]{DCruzVerma}) that for every $\mf m$-primary ideal $I$
in a Cohen-Macaulay local ring of dimension $d$
\[
\mu(I) \leq \eh_{d-1} (\mf m \mid I) + d - 1.
\]
The ideals for which this inequality is actually an equality are said to have {\it minimal mixed multiplicity}.

Thus our condition $(c)$ says that every $\mf m$-full $\mf m$-primary ideal has minimal mixed multiplicity.
\end{remark}

\begin{example}
The assumption of normality of $\mf m$ is important. For example in $\mathbb C[[x,y,z]]/(x^2 + y^2)$ the ideal $I = (x, z^4)$ is integrally closed, but its multiplicity is $8$ while the Loewy length is $5$.
\end{example}

\begin{example}
Let $\mathcal C$ be the class of normal domains of dimension two satisfying the conditions of Corollary \ref{unramified criterion}. Then $\mathcal C$ contains all rational singularities, but a lot more. For example, let $K$ be an algebraically closed field and $R= K[[x,y,z]]/(x^2+g(y,z))$ with $g \in (y,z)^3 - (y,z)^4$, then $R \in \mathcal C$ but $R$ typically does not have rational singularity (see Example 4.3 of \cite{OWY2}). One concrete example is $R= K[[x,y,z]]/(x^2+y^3+z^6)$. 

When $R$ is an excellent normal local domain over an algebraically closed field, then $R \in \mathcal C$  precisely when $\mf m$ is a $p_g$-ideal in the sense of \cite{OWY1}.

\end{example}

\begin{example}
Minimal nonrational double point:  $R= \mathbb C[[x,y,z]]/(x^2 + y^4 + z^4)$. We have  $I= \overline {(x, z^2)} = (x, z^2, yz, z^2)$ has $4$ generators and Loewy length $2$ but multiplicity is $8$ (the reduction number is 2).
\end{example}

\begin{example}
Let $R = k[[x,y,z]]/(x^2 + y^5 + z^5)$ and $I = \overline {(x, z^3)} = (x, y^3, y^2z, yz^2, z^3)$. 
Thus $R$ is a complete local ring with minimal multiplicity and $I$ is an integrally closed $\mf m$-primary ideal.

Observe that 
\[
\eh(I) = \length \left (R/(x,z^3) \right) = \length \left (  k[[x,y,z]]/(x, z^3, y^5) \right) = 15.
\]
On the other hand, $\mu (I) = 5$ and $(x,y,z)^3 \subseteq I$, so
\[
\left (\mu(I) -1 \right ) \loewy(I) < \eh(I).
\]

\end{example}

\subsection{Dimension three}

\begin{definition}
We say that a three-dimensional local algebra $R$ over an algebraically closed field  of characteristic $0$ is  a compound Du Val singularity (cDV) 
if the completion of $R$ is isomorphic to a hypersurface
\[
k[[x,y,z,t]]/\left (f(x,y,z) + tg(x,y,z,t) \right)
\]
where $k[[x,y,z]]/(f(x,y,z))$ is a Du Val  surface (i.e., a rational double point) and $g$ is arbitrary.
\end{definition}

By definition, a cDV singularity has a Du Val singularity as a hyperplane section ($t = 0$).
However, a result of Miles Reid (\cite[Corollary~2.10]{Reid}) asserts that a {\it general}
hyperplane section has a Du Val singularity.

The following corollary is an evidence for Conjecture~\ref{maybe stupid}, since cDV singularities 
are canonical.

\begin{corollary}\label{cool evidence}
Inequality~\ref{ineq} holds if $R$ has a cDV singularity.
\end{corollary}
\begin{proof}
The key ingredient is a result Goto, Iai, and Watanabe in \cite[Proposition~7.5]{GotoIaiWatanabe}
which asserts that every $\mf m$-primary ideal $I$ in a Du Val singularity satisfies the inequality 
$e(I) \leq 2\length (R/I)$ (as if it was a regular local ring). 

Thus we may just follow the proof of Corollary~\ref{loewy general}, see Remark \ref{rmLech}, and use that
for a general linear form $f \in \mf m$ the quotient $R/(f)$ is isomorphic to the rational double point.
\end{proof}

\section{Alternative proofs and related inequalities in dimension two}\label{alternative}

In this section we collect and discuss a few related inequalities in dimension two. For instance, one can give sharper and easier proofs when $R$ is regular. For rational singularites one can give an alternative proof using fiber cones. The relation to the work of in \cite{OWY1, OWY2} will be clarified.  Recall that the order of $I$, $\ord(I)$, is the largest  integer $s$ such that $I\subseteq \m^s$. 

\begin{theorem}\label{regular in dim 2}
Let $(R, \mf m)$ be a two-dimensional regular local ring, then for a full,  $\mf m$-primary ideal $I$
\[
\eh(I) \leq \loewy(I) (\mu(I) - 1) = \loewy(I) \ord (I).
\]
\end{theorem}
\begin{proof}
It is classical, see for instance \cite[Chapter 14]{HS} that $\mu (I) = \ord(I) +1$ if and only if $I$ is $\mf m$-full if and only if $I$ is full. 

Let $k = \loewy(I)$, then $\mf m^k \subseteq I$, then $\mf m^kI^n \subseteq I^{n+1}$, so 
\[
\length (I^n/I^{n+1}) \leq \length(I^n/\mf m^kI^n) = 
\mu (I^n) + \mu (\mf mI^n) + \cdots + \mu (\mf m^{k-1}I^n).
\]
It was shown by Lipman-Zariski (\cite{Lipman}) that the product of full ideals in still full, so 
$\mf m^cI^n$ is full and, thus, 
\[
\mu (\mf m^cI^n) = \ord (\mf m^cI^n) + 1 = n\ord(I) + c + 1
\]
since $\ord$ is a valuation. 
Therefore, 
\[
\length (I^n/I^{n+1}) \leq 
nk \ord(I) + k(k+1)/2
\]
Therefore, as $n$ approaches $\infty$, we obtain that $\eh(I) \leq k (\mu (I) -1)$, which proves the assertion.
\end{proof}

\begin{example}
Using the version above, one can easily see many cases in dimension two when equality occurs. For example, let $R=k[[x,y]]$. Then equality occurs when $I = \overline{(x^a,y^b)}$ or when $I = \overline{(x^a,y^a, x^by^c)}$ with $b+c\leq a$. 

\end{example}

\begin{remark}
The inequality $\eh(I) \leq  \loewy(I) \ord (I)$ does not hold for non-regular local rings.  In the (E8) singularity $R=k[[x,y,z]]/(x^2 + y^3 + z^5)$. Take $I = (x, y^2, yz^2, z^4) = \overline {(x, y^2)}$.
One can check that $\eh(I) = \length (k[[x,y,z]]/(x, y^2, z^5)) = 10$. On the other hand, $\loewy (I) = 4$ and $\ord I = 1$. 
\end{remark}

\begin{remark}
Still assume that $R$ is regular of dimension $2$. We have that $I \subseteq \mf m^{\ord(I)}$, so 
\[
\length (I^n/I^{n+1}) \geq \length (I^n/\mf m^{\ord(I)}I^n) = \mu(I^n) + \mu(\mf mI^n) + \cdots + \mu (\mf m^{\ord (I) - 1} I^n).
\]
Since $\mf m^iI^n \subseteq I^n$ and is $\mf m$-full, 
one can show as above that $\eh(I) \geq \ord (I) (\mu(I) - 1) = \ord(I)^2$.
\end{remark}

Using fiber cones we can give another proof of  our inequality for rational singularities.

\begin{corollary}\label{ratsing}
Let $(R, \mf m)$ be a two-dimensional local rational singularity. Then for any integrally closed $\mf m$-primary ideal $I$
\[
\eh(I) \leq \loewy(I) (\mu(I) - 1).
\]
\end{corollary}
\begin{proof}
It was shown by Lipman that an integrally closed ideal in a rational surface is normal. Hence we may use Theorem~\ref{fiber bound} and obtain that
\[
\eh (I) \leq \eh (\fiber (I)) \loewy(I).
\]
By Lipman-Teissier (\cite[Corollary~5.4]{LipmanTeissier}) any integrally closed ideal has reduction number one. 
Therefore by a result of Shah (\cite[Corollary~7(a)]{Shah}) 
it is known that the fiber cone is Cohen-Macaulay with minimal multiplicity $\mu(I) - 1$
and the claim follows. 
\end{proof}

\begin{remark}\label{OWY methods}
Corollary \ref{ratsing} can be obtained using the geometric techniques of Okuma, Watanabe, and Yoshida (\cite{OWY1, OWY2}). 
If $R$ is normal, we can find a resolution of singularities $X \to \Spec R$ such that 
the ideals $I\sO_X$ and $\mf m\sO_X$ are invertible. In this case, there exist anti-nef cycles $Z, M$ 
supported on the exceptional divisor of $X$
such that $I\sO_X = \sO_X (-Z)$ and $\mf m\sO_X = \sO_X (-M)$.
Under this assumptions we can compute multiplicity as a self-intersection number, $\eh(I) = - Z^2$.

In the notation of \cite{OWY1}, the statement asserts that our inequality holds 
for all integrally closed $\mf m$-primary ideals $I$ if $\mf m$ is a $p_g$-ideal.
In this case, as remarked in \cite[Example~6.2]{OWY1},  $\mu (I) = -MZ + 1$ where $MZ$ is the intersection number.

By the definition of $\loewy (I)$,  $\loewy (I) M - Z$ is effective, so 
we have $Z (Z - \loewy (I) M) \geq 0$ because $Z$ is anti-nef. 
Thus,  
\[
\loewy(I)(\mu(I) - 1) = -\loewy(I)MZ \geq -Z^2 = \eh (I).
\]
\end{remark}

\begin{remark}
Lemma \ref{numgen bound} can also be obtained using the techniques of \cite{OWY1}. 

In \cite[Theorem~6.1]{OWY1} it was shown that 
\[
\length (I/\overline{\mf mI}) = -MZ + 1 - \length \left( \frac{\overline{\mf mI}}{a\mf m + xI}\right)
\]
where $a \in I, x \in \mf m$ are general elements and $MZ$ is the intersection number.
Of course, $\mu (I) = \length (I/\overline{\mf m I}) + \length (\overline{\mf mI}/\mf mI)$, so 
\[
\mu(I) = -MZ + 1 - \length \left( \frac{\mf m I}{a\mf m + xI}\right) \leq - MZ + 1.
\]
Therefore, if $I = \overline{\mf m^n}$ and, thus, $Z = nM$, we obtain that 
\[
\mu(I) \leq -nM^2 + 1 = n\eh(\mf m) + 1.
\]
\end{remark}

\begin{remark}
Next, we discuss another related inequality. If $I$ is an $\mf m$-full ideal then we know that $\mu (I) \geq \mu(J)$ for any $I \subseteq J$. 
Observe that $I \subseteq \mf m^{\ord I}$, so $\mu(I) \geq \mu (\mf m^{\ord I})$.

If $R$ has minimal multiplicity we know that 
\[
\mu (\mf m^{\ord I}) = \eh(R) \binom{\ord(I) + d - 2}{d - 1} + \binom{\ord(I) + d - 2}{d - 2}.
\]

In dimension $2$ the formula above tells us that $\mu (\mf m ^{\ord I}) = \eh(R)\ord(I) + 1$, 
so  $$\mu(I) \geq \eh(R) \ord (I) + 1,$$ in other words $(\mu (I) - 1) \loewy(I) \geq \eh(R) \ord(I) \loewy(I)$.
Therefore, one can study the inequality:
\[\eh(R) \ord(I) \loewy(I) \geq \eh(I).\]

In general it could be a finer inequality then inequality \ref{ineq}. Note that this holds for all ideals in a regular local ring of dimension $2$ by Theorem \ref{regular in dim 2}.
To show this just observe that $\eh(I) = \eh(\bar{I})$, but $\loewy (\bar{I}) \leq \loewy(I)$ and $\ord(\bar{I}) \leq \ord(I)$ because $I \subseteq \bar{I}$.
\end{remark}

\section{A conjectured inequality on $F$-thresholds}\label{threshold}
In this section we discuss a conjecture posed by Huneke, Musta\c{t}\v{a}, Takagi, and Watanabe (\cite[Conjecture~5.1]{HMTW}).  We shall show that a special case of this conjecture in dimension two follows from inequality \ref{ineq}.  This result was suggested by a conversation with Mircea Musta\c{t}\v{a} after a talk given by the first author at the University of Michigan.

Before we state the conjecture, let us recall that the F-threshold of an ideal $\mf a$ with respect to a parameter ideal
$J = (x_1, \ldots, x_d)$ is defined as the limit (the fact that the limit exists was settled fairly recently in \cite{DNP}):
\[
\thj^{J} (\mf a) = \lim_{e \to \infty} \frac{\min\{ N \mid \mf a^N \subseteq (x_1^{p^e}, \ldots, x_d^{p^e})\}}{p^e}.
\]

\begin{conjecture}\label{HMTW}
Let $(R, \mf m)$ be a $d$-dimensional Noetherian local ring of characteristic $p > 0$. If $J \subseteq \mf m$
is an ideal generated by system of parameters and $\mf a$ is an $\mf m$-primary ideal, then 
\[
\eh(J) \leq \left(\frac{\thj^J (\mf a)}{d} \right)^d \eh(\mf a).
\]
\end{conjecture}

The characteristic-free version of this conjecture has appeared in \cite[Conjecture~2.6]{HTW}. When $R$ is regular and $\mf a=\m$, a special case of the conjecture, which is still open, would read: 

\begin{conjecture}\label{HTW}
Let $(R, \mf m)$ be a $d$-dimensional  regular local ring. If $J \subseteq \mf m$
is an ideal generated by system of parameters  then 
\[
\eh(J) \leq \left(\frac{\thj^J (\mf m)}{d} \right)^d= \left(\frac{d + \loewy (J) - 1}{d} \right)^d.
\]

\end{conjecture}

We will show that a special case of the conjecture follows from Inequality~\ref{ineq}.
The proof is adapted from \cite[Theorem~3.3]{HMTW}. Note that the proof of that result contains a mistake, as explained in  in the footnote on the page 132 of 
\cite{HTW}.

\begin{proposition}\label{crll}
Let $(R, \m)$ be regular local ring of dimension $d$ and characteristic $p > 0$. If $J \subseteq \mf m$ is an ideal generated by system of parameters, then
$$ \loewy(\overline{J}) \leq \bceil{\frac{\thj^J(\m) -\ord(J)}{d-1}}$$

\end{proposition}

\begin{proof}
By \cite[Example~2.7(iii)]{HMTW}, $c = \thj (\mf m) = \loewy (J) + 1$.
By \cite[Proposition~2.2 (vii)]{HMTW} this gives us that for all $e_0$ and all sufficiently large $e$ (depending on $e_0$), 
\[
\mf m^{p^e(c + p^{-e_0})} \subseteq J^{[p^e]}.
\]
Let $r = \ord(J)$ and observe an induced inclusion
\[
J^{p^e}\mf m^{p^e(c - r + p^{-e_0})} \subseteq J^{[p^e]}.
\]
By \cite[Example~2.7(i)]{HMTW}
\[
\mf m^{p^e(c-r + p^{-e_0})} \subseteq J^{[p^e]} : J^{p^e} \subseteq J^{(d-1)(p^e - 1)}.
\]
Let $\nu$ a discrete valuation centered at $\mf m$, then from the containment we have 
\[
p^e(c - r + p^{-e_0}) \nu(\mf m) \geq \left (d-1)(p^e - 1 \right) \nu(J).
\]
Dividing by $p^e$ and taking the limit gives
\[
(c - r + p^{-e_0}) \nu(\mf m) \geq (d-1)\nu(J).
\]
Now, we may let $e_0$ go to infinity and obtain that $(c- r) \nu (\mf m) \geq (d-1)\nu(J)$. Since
$\nu$ is arbitrary, we obtain that
\[
\mf m^{\ceil{\frac{c - r}{d-1}}} \subseteq \overline{J}.
\]
Thus $\loewy (\overline{J}) \leq \ceil{\frac{c - r}{d-1}}$.

\end{proof}

\begin{remark}
Following the proof of \cite[Theorem 3.3]{HMTW}, one can prove this result for non-regular rings using tight-closure. 
\end{remark}

\begin{theorem}
Let $(R, \mf m)$ be a two dimensional regular local ring of characteristic $p > 0$.  
If $J \subseteq \mf m$ is an ideal generated by system of parameters, then
\[
\eh(J) \leq \left(\frac{\thj^J(\m)}{2}\right)^2  = \frac{(\loewy (J) + 1)^2}{4}.
\]
\end{theorem}

\begin{proof}
Let $c= \thj^J(\m)$ and  $r= \ord(J)$. Using the version of inequality \ref{ineq} stated in Theorem~\ref{regular in dim 2}, we have
\[
\eh(J) = \eh(\overline J) \leq \loewy(\overline{J})r \leq (c - r) r \leq \frac{c^2}{4},
\]
where  the second inequality follows by Proposition \ref{crll}. 
\end{proof}

We note also the following relative version of Inequality~\ref{ineq}, which might be relevant for the general case of Conjecture \ref{HMTW}.  
We will use $\loewy_J (I) = \min \{N \mid J^N \subseteq I\}$.

\begin{corollary}
Let $(R, \mf m)$ be a two-dimensional rational singularity and $I, J$ be two integrally closed $\mf m$-primary ideals.
Then 
\[
\eh(I) \leq \loewy_J (I) \left(\length (I/IJ) - \length (R/J) \right).
\]
\end{corollary}
\begin{proof}
By definition, $\eh(I) \leq \eh(I \mid J^{\loewy_J (I)})$. The result now follows 
from Lemma~\ref{Rees lemma}, since every integrally closed ideal in a rational singularity is normal and has reduction number $1$ (\cite{LipmanTeissier}).

\end{proof}

\section{The asymptotic inequality and fiber cones}\label{fiber}

In this section we focus on asymptotic versions of Inequality~\ref{ineq} and its relationship to the multiplicity of the fiber cone.  

\begin{definition}
Let $I$ be an ideal in a local ring $(R, \mf m)$. The fiber cone of $I$ is the graded ring 
\[
\fiber(I) = \bigoplus_{n \geq 0} I^n/\mf mI^n
\]
and the normal fiber cone of $I$ is defined as the graded ring 
\[
\overline{\fiber}(I) = \bigoplus_{n \geq 0} \overline{I^n}/\mf m\overline{I^n}.
\]

Alternatively, the fiber cone is closed fiber of the Rees algebra $R[It]$ and the normal fiber cone
is the closed fiber of its normalization in $R[t]$.
\end{definition}

\begin{observation}\label{fiber observation}
Suppose that Inequality~\ref{ineq} holds for large powers of an ideal $I$. Then 
\[
\eh(I) \leq \frac{(d-1)! \mu(I^n) \loewy(I^n)}{n^d},
\]
and in the limit we obtain that 
\[
\eh(I) \leq \eh (\fiber(I)) \lim_{n \to \infty} \frac{\loewy(I^n)}{n} \leq \eh(\fiber(I)) \loewy(I),
\]
since $\loewy (I^n) \leq n \loewy(I)$.

Alternatively, we may apply Inequality~\ref{ineq} to the integral closures $\overline{I^n}$ 
and obtain another version:
\[
\overline{\eh} (I) \leq \eh(\overline{\fiber}(I)) \loewy(I),
\]
where we have used $\overline{\eh}(I)$ to denote $\lim\limits_{n \to \infty} \frac{d! \length (R/\overline{I^n})}{n^d}$.
\end{observation}

The next theorem asserts that the second asymptotic inequality always holds. 

\begin{theorem}\label{fiber bound}
Let $(R, \mf m)$ be a local ring of positive dimension and $I$ be an $\mf m$-primary ideal. 
Then 
\[
\overline{\eh} (I) \leq \eh (\overline{\fiber} (I)) \loewy(I).
\]
\end{theorem}
\begin{proof}
Let $k = \loewy (I)$ and observe that
\[
\length \left (\overline{I^n}/\overline{I^{n+1}} \right) \leq 
\length \left (\overline{I^n}/\overline{\mf m^kI^n} \right) 
= \length \left (\overline{I^n}/\overline{\mf mI^n} \right ) + \cdots + 
\length \left( \overline{\mf m^{k-1}I^n}/\overline{\mf m^kI^n} \right).
\]
Observe that 
\[
\length \left (\overline{\mf m^{i}I^n}/\overline{\mf m^{i+1}I^n} \right) \leq 
\length \left (\overline{\mf m^{i}I^n}/\mf m(\overline{\mf m^{i}I^n}) \right)
= \mu  \left(\overline{\mf m^iI^n} \right).
\]
Since $\dim R > 0$, $\overline{I^n} \neq \sqrt{0}$ so it is $\mf m$-full by \cite[Theorem~2.4]{Goto}.
Thus $\mu (\overline{I^{n + i}}) \geq \mu (\overline{\mf m^i I^n})$ because $I^{n + i} \subseteq \mf m^i I^n$.
Hence we obtain that 
\[
\length \left (\overline{I^n}/\overline{I^{n+1}} \right) \leq 
\mu \left (\overline{I^n} \right) + \cdots + \mu \left (\overline{I^{n+k - 1}} \right).
\]

Observe that $\dim R = \dim \fiber(I)$ since $I$ is $\mf m$-primary.
Thus
\[
\overline{\eh}(I) = \lim_{n \to \infty} \frac{(d-1)! \length (\overline{I^n}/\overline{I^{n+1}})}{n^{d-1}}
\leq \lim_{n \to \infty} \frac{(d-1)! \left(\mu( \overline{I^n}) + \cdots + \mu (\overline{I^{n+k - 1}}) \right) }{n^{d-1}}
= k \eh(\overline{\fiber}(I)).
\] 
\end{proof}

\begin{corollary}
Let $(R, \mf m)$ be a local ring of positive dimension and $I$ be a normal $\mf m$-primary ideal. 
Then 
\[
\eh (I) \leq \eh (\fiber (I)) \loewy(I).
\]
\end{corollary}

\begin{corollary}
Let $(R, \mf m)$ be an analytically unramified local ring of positive dimension and $I$ be an $\mf m$-primary ideal. 
Then
\[
\eh (I) \leq \eh (\overline{\fiber} (I)) \loewy(I).
\]
\end{corollary}
\begin{proof}
Since $R$ is analytically unramified, by a classical result of Ratliff (\cite[Theorem~2.7, Corollary~4.5]{Ratliff}) 
there exists $s > 0$ such that $\overline{I^{ns}} = (\overline {I^s})^n$ for all $n$,
i.e., $\overline{I^s}$ is normal.
Hence 
\[
s^{\dim R}\overline{\eh}(I) = \overline{\eh(I^s)} = 
\eh(I^s) = s^{\dim R}\eh(I)
\]
and the corollary follows after noting that $\loewy (I) \geq \loewy (\overline{I})$.
\end{proof}

\begin{remark}
It is worth noting that Ratliff's result also shows that $\eh(\fiber(I)) \leq \eh(\overline{\fiber}(I))$.
Namely, since $I^s$ is a reduction of $\overline{I^s}$, for some constant $c$ we have 
\[
(\overline{I^s})^n = I^{s(n-c)} (\overline{I^s})^c \subseteq I^{s(n-c)}.
\]
Thus, by Remark~\ref{pass m-full} $\mu ((\overline{I^s})^n) \geq \mu (I^{s(n-c)})$ and
\[
\eh(\overline{\fiber}(I)) = \lim_{n \to \infty} \frac{\mu (\overline{I^{sn}} ) }{(ns)^{\dim R - 1}}
\geq \lim_{n \to \infty} \frac{\mu (I^{s(n-c) } ) }{(ns)^{\dim R - 1}} = \eh (\fiber(I)).
\]
\end{remark}

\begin{corollary}
Let $(R, \mf m)$ be an analytically unramified local ring of dimension $d > 0$. 
Then for any integrally closed $\mf m$-primary ideal $I$
there exists an integer $s$ such that
\[
\eh (I) \leq \frac{1}{s^{d-1}} \eh\left (\fiber(\overline{I^s}) \right ) \loewy (I).
\]
\end{corollary}
\begin{proof}
As above, we choose $s$ so that $\overline{I^s}$ is normal.
Applying Theorem~\ref{fiber bound} to $\overline{I^s}$, we obtain the inequality
\[
\eh(\overline{I^s}) \leq \eh \left( \fiber(\overline{I^s})\right ) \loewy (\overline{I^s}).
\]
Now, the statement easily follows after observing that $\eh (\overline {I^s}) = \eh(I^s) = s^d \eh (I)$
and $$\loewy(\overline {I^s}) \leq \loewy (I^s) \leq s \loewy(I).$$
\end{proof}

\section{Powers of the maximal ideal and h-vectors}\label{powers}
In this section we will study the restrictions on the singularity obtained 
by requiring Inequality~\ref{ineq} to hold for all (large) powers of the maximal ideals. 

\begin{proposition}
Let $(R, \mf m)$ be a Cohen-Macaulay local ring of dimension $2$. 
Then the following conditions are equivalent:
\begin{enumerate}[(a)]
\item $R$ has minimal multiplicity,
\item $\mf m^n$ satisfy inequality~\ref{ineq} for all $n$, \label{all}
\item $\mf m^n$ satisfy inequality~\ref{ineq} for infinitely many $n$. \label{some}
\end{enumerate}
\end{proposition}
\begin{proof}

If $R$ has minimal multiplicity, then necessarily the Hilbert function is 
\[
\mu (\mf m^n) = \eh(R)n + 1 
\]
and the inequality follows for all powers.

Clearly, $\ref{all} \Rightarrow \ref{some}$, so we establish the last implication.
Inequality~\ref{ineq} applied to $\mf m^n$ gives the formula
\[
n^2 \eh(R) = \eh(\mf m^n) \leq n \left (\mu (\mf m^n) - 1 \right).
\] 
Let the Hilbert polynomial of $R$ be $\mu (\mf m^n) = \eh(R)n + c$.
Hence, if the inequality holds for infinitely many $n$, then we must have $c \geq 1$. 

It is not hard to see that if we express the Hilbert-Samuel function as
\[
\length (R/I^{n+1}) = \eh(I) \binom {n + 2}{2} - \eh_1 \binom {n+1}{1} + \eh_2
\]
then $c = \eh - \eh_1$. Thus we must have that $\eh \geq \eh_1 + 1$. 

On the other hand, the celebrated inequality of Northcott (\cite[Theorem~1]{Northcott})
asserts that in a Cohen-Macaulay ring $\eh \leq \eh_1 + 1$. 
Hence $\eh = \eh_1 + 1$ and, by \cite[Corollary~2.2]{EliasValla}, this is equivalent to the minimal multiplicity of $R$.

\end{proof}

However, all integrally closed ideals in a ring of dimension two with minimal multiplicity may not satisfy Inequality~\ref{ineq}, as we know from Theorem \ref{mainTheorem}.  It is easy to see that one direction holds in all dimensions.

\begin{proposition}
Let $(R, \mf m)$ be a Cohen-Macaulay local ring of dimension $d$ that has minimal multiplicity.
Then the powers of the maximal ideal satisfy inequality~\ref{ineq}.
\end{proposition}
\begin{proof}
By \cite[Theorem~1]{Sally2} the Hilbert function of $R$ can be written as
\[
\mu (\mf m^n) = \eh(R) \binom {n + d -2}{d - 1} + \binom{n + d -2}{d - 2}. 
\]
Since $n \geq 1$, $\binom{n + d - 2}{d-2} \geq d - 1$, so we have
\[
(d-1)! n\left( \mu(\mf m^n) - d + 1\right) \geq (d-1)! n\eh(R) \binom {n + d -2}{d - 1} \geq n^{d} \eh(R) =  \eh (\mf m^n).
\]
\end{proof}

\begin{remark}
By looking at the Hilbert functions  one can see, as we do below,  that minimal multiplicity is not necessary in dimension at least $3$. 
\end{remark}

Now let us investigate more carefully what happens in higher dimension. For the powers of maximal ideal, the problem immediately reduces to the graded case, so we 
assume $R$ is a standard graded $k$-algebra.  Suppose  that the Hilbert series of $R$ is given by
\[
H_R(t) = \frac{\sum_i a_it^i}{(1 - t)^d} = \left (\sum_{i = 0}^h a_it^i\right) \left( \sum_i \binom {i + d-1}{d - 1}t^i\right),
\]
so $(a_0, a_1, \ldots, a_h)$ is the $h$-vector.
Thus we have
\[
\mu(\mf m^c) = \sum_{i=c -h}^c \binom{i + d - 1}{d - 1}a_{c - i} =  p(c).
\]

Our inequality becomes:
\[
(d-1)!(p(c)-d+1)\geq \eh(R)c^{d-1}
\]
One can calculate the coefficients of  $p(c)$ for $c\gg 0$.  For instance, the first coefficient is $\frac{\sum a_i}{(d-1)!} = \frac{\eh(R)}{(d-1)!}$, so we can see that inequality $1$ is tight asymptotically.  Similarly, we can write the second coefficient (at $c^{d-2}$) of the polynomial $p(c)$ as 
\begin{align*}
&(1+...+(d-1))a_0 + (0 +...+(d-2))a_1  + (-1 +...+(d-3))a_2 + \cdots + ((1 - h) +...+(d-1 - h))a_h\\
&= \sum_{i = 0}^h \left ( \frac{(d - 1)d}{2} - (d - 1)i \right ) a_i.
\end{align*}

So asymptotically, we need this coefficient to be non-negative, etc. We summarize our findings in:

\begin{proposition} 
Let $(R,\m)$ be a local ring and let $(a_0, a_1, \dots, a_h)$ be the $h$-vector of the associated graded ring of $\m$. 
\begin{enumerate}
\item Inequality $1$ holds for all large powers of $\m$ if and only if $\sum_{i = 0}^h \left ( \frac{(d - 1)d}{2} - (d - 1)i \right ) a_i\geq 0$. 
\item Inequality $1$ holds for all large powers of $\m$ if $a_i \geq a_{d-i}$ for all $i\leq \frac{d}{2}$ (this implies that $a_i=0$ for $i>d$). 
\item When $d=3$, inquality $1$ becomes an equality for all large powers of $\m$ if $a_i= a_{3-i}$ for all $i$. 
\end{enumerate}

\end{proposition}

\begin{proof}
This is straightforward from the preceding discussion and some elementary algebra. 
\end{proof}

\begin{remark}
When $c$ is small, inequality~\ref{ineq} forces more stringent conditions on the singularity. For example, when $c=1$, we must have 
$$(d-1)!(n-d+1)\geq \eh(R)$$
here $n$ is the embedding dimension of $R$. We do not know if this holds for a large class of singularities. The most reasonable guess would be Gorenstein rational or terminal singularities. Note that similar  bounds have been proven for rational sinularities by Huneke-Watanabe in \cite{HW}. However, for $n$ large those bounds are weaker than the inequality above. One can also use examples from {\it loc. cit.} to see that for our inequality to hold for all integrally closed ideals, assuming $R$ to have rational singularities (but without being Gorenstein) is not enough.   
\end{remark}

\bibliographystyle{plain}
\bibliography{bounds}

\begin{thebibliography}{10}

\bibitem{Abhyankar}
Shreeram~Shankar Abhyankar.
\newblock Local rings of high embedding dimension.
\newblock {\em Amer. J. Math.}, 89:1073--1077, 1967.

\bibitem{DaoTakahashi}
Hailong Dao and Ryo Takahashi.
\newblock Upper bounds for dimensions of singularity categories.
\newblock {\em C. R. Math. Acad. Sci. Paris}, 353(4):297--301, 2015.

\bibitem{DCruzVerma2}
Clare D'Cruz and J.~K. Verma.
\newblock On the number of generators of {C}ohen-{M}acaulay ideals.
\newblock {\em Proc. Amer. Math. Soc.}, 128(11):3185--3190, 2000.

\bibitem{DCruzVerma}
Clare D'Cruz and J.~K. Verma.
\newblock Hilbert series of fiber cones of ideals with almost minimal mixed
  multiplicity.
\newblock {\em J. Algebra}, 251(1):98--109, 2002.

\bibitem{DNP}
Alessandro {De Stefani}, Luis {N{\'u}{\~n}ez-Betancourt}, and Felipe
  {P{\'e}rez}.
\newblock On the existence of $f$-thresholds and related limits.
\newblock Preprint, available at http://arxiv.org/abs/1605.03264.

\bibitem{EliasValla}
Juan Elias and Giuseppe Valla.
\newblock Rigid {H}ilbert functions.
\newblock {\em J. Pure Appl. Algebra}, 71(1):19--41, 1991.

\bibitem{Goto}
Shiro Goto.
\newblock Integral closedness of complete-intersection ideals.
\newblock {\em J. Algebra}, 108(1):151--160, 1987.

\bibitem{GotoIaiWatanabe}
Shiro Goto, Sin-Ichiro Iai, and Kei-Ichi Watanabe.
\newblock Good ideals in {G}orenstein local rings.
\newblock {\em Trans. Amer. Math. Soc.}, 353(6):2309--2346, 2001.

\bibitem{HIO}
Manfred Herrmann, Shin Ikeda, and Ulrich Orbanz.
\newblock {\em Equimultiplicity and blowing up}.
\newblock Springer-Verlag, Berlin, 1988.
\newblock An algebraic study, With an appendix by B. Moonen.

\bibitem{Huneke}
Craig Huneke.
\newblock Hilbert functions and symbolic powers.
\newblock {\em Michigan Math. J.}, 34(2):293--318, 1987.

\bibitem{HMTW}
Craig Huneke, Mircea Musta{\c t}{\u a}, Shunsuke Takagi, and Kei-ichi Watanabe.
\newblock F-thresholds, tight closure, integral closure, and multiplicity
  bounds.
\newblock {\em Michigan Math. J.}, 57:463--483, 2008.
\newblock Special volume in honor of Melvin Hochster.

\bibitem{HS}
Craig Huneke and Irena Swanson.
\newblock {\em Integral closure of ideals, rings, and modules}, volume 336 of
  {\em London Mathematical Society Lecture Note Series}.
\newblock Cambridge University Press, Cambridge, 2006.

\bibitem{HTW}
Craig Huneke, Shunsuke Takagi, and Kei-ichi Watanabe.
\newblock Multiplicity bounds in graded rings.
\newblock {\em Kyoto J. Math.}, 51(1):127--147, 2011.

\bibitem{HW}
Craig Huneke and Kei-ichi Watanabe.
\newblock Upper bound of multiplicity of {F}-pure rings.
\newblock {\em Proc. Amer. Math. Soc.}, 143(12):5021--5026, 2015.

\bibitem{Itoh}
Shiroh Itoh.
\newblock Integral closures of ideals generated by regular sequences.
\newblock {\em J. Algebra}, 117(2):390--401, 1988.

\bibitem{Lech}
Christer Lech.
\newblock Note on multiplicities of ideals.
\newblock {\em Ark. Mat.}, 4:63--86 (1960), 1960.

\bibitem{Lipman}
Joseph Lipman.
\newblock Rational singularities, with applications to algebraic surfaces and
  unique factorization.
\newblock {\em Inst. Hautes \'Etudes Sci. Publ. Math.}, (36):195--279, 1969.

\bibitem{LipmanTeissier}
Joseph Lipman and Bernard Teissier.
\newblock Pseudorational local rings and a theorem of {B}rian\c con-{S}koda
  about integral closures of ideals.
\newblock {\em Michigan Math. J.}, 28(1):97--116, 1981.

\bibitem{Northcott}
D.~G. Northcott.
\newblock A note on the coefficients of the abstract {H}ilbet function.
\newblock {\em J. London Math. Soc.}, 35:209--214, 1960.

\bibitem{OWY1}
Tomohiro Okuma, Kei-ichi Watanabe, and Ken-ichi Yoshida.
\newblock Good ideals and {$p_g$}-ideals in two-dimensional normal
  singularities.
\newblock {\em Manuscripta Math.}, 150(3-4):499--520, 2016.

\bibitem{OWY2}
Tomohiro Okuma, Kei-ichi Watanabe, and Ken-ichi Yoshida.
\newblock Rees algebras and {$p_g$}-ideals in a two-dimensional normal local
  domain.
\newblock {\em Proc. Amer. Math. Soc.}, 145(1):39--47, 2017.

\bibitem{Ratliff}
L.~J. Ratliff, Jr.
\newblock Notes on essentially powers filtrations.
\newblock {\em Michigan Math. J.}, 26(3):313--324, 1979.

\bibitem{ReesRat}
David Rees.
\newblock Hilbert functions and pseudorational local rings of dimension two.
\newblock {\em J. London Math. Soc. (2)}, 24(3):467--479, 1981.

\bibitem{ReesSharp}
David Rees and Rodney Sharp.
\newblock On a theorem of {B}. {T}eissier on multiplicities of ideals in local
  rings.
\newblock {\em J. London Math. Soc. (2)}, 18(3):449--463, 1978.

\bibitem{Reid}
Miles Reid.
\newblock Canonical {$3$}-folds.
\newblock In {\em Journ\'ees de {G}\'eometrie {A}lg\'ebrique d'{A}ngers,
  {J}uillet 1979/{A}lgebraic {G}eometry, {A}ngers, 1979}, pages 273--310.
  Sijthoff \&\ Noordhoff, Alphen aan den Rijn---Germantown, Md., 1980.

\bibitem{Sally2}
Judith~D. Sally.
\newblock Cohen-{M}acaulay local rings of maximal embedding dimension.
\newblock {\em J. Algebra}, 56(1):168--183, 1979.

\bibitem{Shah}
Kishor Shah.
\newblock On the {C}ohen-{M}acaulayness of the fiber cone of an ideal.
\newblock {\em J. Algebra}, 143(1):156--172, 1991.

\bibitem{Teissier}
Bernard Teissier.
\newblock Cycles \'evanescents, sections planes et conditions de {W}hitney.
\newblock In {\em Singularit\'es \`a {C}arg\`ese ({R}encontre {S}ingularit\'es
  {G}\'eom. {A}nal., {I}nst. \'{E}tudes {S}ci., {C}arg\`ese, 1972)}, pages
  285--362. Ast\'erisque, Nos. 7 et 8. Soc. Math. France, Paris, 1973.

\bibitem{VermaComplete}
J.~K. Verma.
\newblock Joint reductions of complete ideals.
\newblock {\em Nagoya Math. J.}, 118:155--163, 1990.

\bibitem{WatanabeJ}
Junzo Watanabe.
\newblock {${\mf m}$}-full ideals.
\newblock {\em Nagoya Math. J.}, 106:101--111, 1987.

\end{thebibliography}

\end{document}